\documentclass[11pt]{article}%[10pt]{article}%\documentclass
\usepackage{amsmath}
\usepackage{fancyvrb}
\usepackage[all]{xy}
\usepackage{tabularx}
\usepackage{amscd}
\usepackage{mhchem}
\usepackage{amsthm}
\usepackage{amsfonts}

\usepackage{hyperref}
\usepackage{amssymb}
\usepackage{graphicx}

\def\Hom{\mathop{\rm Hom}\nolimits}
\def\Ext{\mathop{\rm Ext}\nolimits}

\def\mod{\mathop{\rm mod}\nolimits}

\def\End{\mathop{\rm End}\nolimits}
\def\tilt{\mathop{\rm tilt}\nolimits}

\def\proj{\mathop{\rm proj}\nolimits}

\def\silt{\mathop{\rm silt}\nolimits}

\usepackage{ulem}
\begin{document}

%% ----------------------------------------------------------------------
\newcommand{\nc}{\newcommand}
\def\PP#1#2#3{{\mathrm{Pres}}^{#1}_{#2}{#3}\setcounter{equation}{0}}
\def\ns{$n$-star}\setcounter{equation}{0}
\def\nt{$n$-tilting}\setcounter{equation}{0}
\def\Ht#1#2#3{{{\mathrm{Hom}}_{#1}({#2},{#3})}\setcounter{equation}{0}}
\def\qp#1{{${(#1)}$-quasi-projective}\setcounter{equation}{0}}
\def\mr#1{{{\mathrm{#1}}}\setcounter{equation}{0}}
\def\mc#1{{{\mathcal{#1}}}\setcounter{equation}{0}}
\def\HD{\mr{Hom}_{\mc{D}}}
\def\HC{\mr{Hom}_{\mc{C}}}
\def\AdT{\mr{Add}_{\mc{T}}}
\def\adT{\mr{add}_{\mc{T}}}
\def\Kb{\mc{K}^b(\mr{Proj}R)}
\def\kb{\mc{K}^b(\mc{P}_R)}
\def\AdpC{\mr{Adp}_{\mc{C}}}

\newtheorem{theorem}{Theorem}[section]
\newtheorem{proposition}[theorem]{Proposition}
\newtheorem{lemma}[theorem]{Lemma}
\newtheorem{corollary}[theorem]{Corollary}
\newtheorem{conjecture}[theorem]{Conjecture}
\newtheorem{question}[theorem]{Question}
\newtheorem{definition}[theorem]{Definition}
\newtheorem{example}[theorem]{Example}

\newtheorem{remark}[theorem]{Remark}

\def\Pf#1{{\noindent\bf Proof}.\setcounter{equation}{0}}
\def\>#1{{ $\Rightarrow$ }\setcounter{equation}{0}}
\def\<>#1{{ $\Leftrightarrow$ }\setcounter{equation}{0}}
\def\bskip#1{{ \vskip 20pt }\setcounter{equation}{0}}
\def\sskip#1{{ \vskip 5pt }\setcounter{equation}{0}}

\def\bg#1{\begin{#1}\setcounter{equation}{0}}
\def\ed#1{\end{#1}\setcounter{equation}{0}}
\def\KET{T^{^F\bot}\setcounter{equation}{0}}
\def\KEC{C^{\bot}\setcounter{equation}{0}}
\renewcommand{\thefootnote}{\fnsymbol{footnote}}
\setcounter{footnote}{0}

\title{\bf Extending silted algebras to cluster-tilted algebras 
\thanks{This work was partially supported by NSFC (Grant No. 11971225). } }
\footnotetext{
E-mail:~hpgao07@163.com}
\smallskip
\author{\small Hanpeng Gao \thanks{Corresponding author.}\\
{\it \footnotesize Department of Mathematics, Nanjing University, Nanjing 210093, P.R. China
}\\
{\it \footnotesize Department of Mathematics,  University of Connecticut, Storrs 06269, USA}}

\date{}
\maketitle
\baselineskip 15pt\begin{abstract}
It  is well known that the relation-extensions of  tilted algebras are cluster-tilted algebras.  In this paper, we extend the result to silted algebras and prove some extension
 of silted algebras are cluster-tilted algebras. %\mskip\
\vspace{10pt}


%\sskip\


\end{abstract}
%\smallskip

\vskip 30pt
% ----------------------------------------------------------------------
%% ----------------------------------------------------------------------
%\def\baselinestretch{1}

\section{Introduction}
Cluster-tilted algebras were introduced by Buan, Marsh, Reiten and etl. \cite{BMRRT} , and also in \cite{CCS} for type $\mathbb{A}$.  Let $A$ be a triangular algebra whose global dimension is at most two over an algebraically closed field $k$. The trivial extension of $A$ by the $A$-$A$-bimodule $\Ext^{2}_A(DA,A)$ is called the $relation$-$extension$\cite{ABS} of $A$, where $D$=$\Hom_k(-,k)$ is the standard duality. It is proved that the relation-extension of every tilted algebra is cluster-tiled , and every cluster-tilted algebra is of this form in \cite{ABS}.

The concept of silting complexes originated from \cite{KV} and    2-term silting complexes are of particular interest and important for the representation of algebra. 
 In \cite{BZ2018}, the endomorphism algebras of 2-term silting complexes were introduced by Buan and Zhou.  They also defined the concept of  the silted algebra \cite{BZ2016} which is the endomorphism algebras of 2-term silting complex over the derived category of hereditary algebras and proved that an algebra is silted if and only if it is  shod \cite{CL1999}( projective dimension or injective dimension of every indecomposable module is at most one).  In particular, tilted algebras are silted, indeed,   the minimal projective presentation of a tilting module $T$ over the hereditary algebra $H$ gives rise to a 2-term silting complex  $\mathbf{ P}$ in $\mathrm{K}^b(\proj H)$,  and. that there is an isomorphism of algebras $\End_H(T ) \cong \End_{\mathcal{D}^b(H)}(\mathbf{ P })$.

As a generalization of tilting modules, support $\tau$-tilting modules were introduced by Adachi, Iyama and Reiten \cite{AIR}. They also shown that there is a bijection between support $\tau$-tilting modules and 2-term silting complexes(see, \cite[Theorem 3.2]{AIR}).   This result provided that every silted algebra can be described as the triangular matrix algebra  $\left(\begin{array}{cc} B & 0 \\
M & H_1\end{array}\right)$ where $B$ is a tilted algebra, $H_1$ is a hereditary algebra and $M$ is a $H_1$-$B$-bimodule (see, Proposition \ref{3.2}). It is a natural question whether silted algebras can be extended to cluster-tilted algebras.

In this paper, we give a positive answer and  construct   cluster-tilted algebras from silted algebras.   We call a silted algebra $A$ with respect to  $(T, P)$ for some hereditary algebra $H$ if there exists a 2-term silting complex $\mathbf{ P}$ in $\mathcal{D}^b(H)$ which corresponding to the support $\tau$-tilting pair  $(T,P)$ in $\mod H$ such that $A\cong \End_{\mathcal{D}^b(H)}(\mathbf{ P })$.
Our main results as follows.
\begin{theorem}
Let $A$=$\left(\begin{array}{cc} B & 0 \\
M & H_1\end{array}\right)$ be a silted algebra with respect to  $(T, P)$ for some hereditary algebra $H$. Then the matrix algebra $\left(\begin{array}{cc} B\ltimes\Ext^1_H(T,\tau^{-1}_HT) & \Hom_H(P, \tau^{-1}_H T) \\
 M & H_1\end{array}\right)$  is a cluster-tilted algebra.
\end{theorem}
As a consequence, we have the following result.
\begin{theorem}
Let $A$=$\left(\begin{array}{cc} B & 0 \\
M & H_1\end{array}\right)$ be a silted algebra with respect to  $(T, P)$ for some hereditary algebra $H$. If   $\Hom_H(P,\tau^{-1}_HT)=0$, then the triangular matrix algebra $\left(\begin{array}{cc} B\ltimes\Ext^2_B(DB,B) & 0 \\
M & H_1\end{array}\right)$ is a cluster-tilted algebra.
\end{theorem}

Note that a tilted algebra is exactly silted algebra with respect to  $(T, 0)$ for some hereditary algebra $H$, we can easy get the relation-extension of every tilted algebra is cluster-tilted .

Throughout this paper, all algebras are finite dimensional algebras over an algebraically closed field $k$. For an algebra $A$, we denoted by $\mod A$ the category of finitely generated right $A$-modules and $\proj A$  the category of finitely generated projective right $A$-modules.  $\mathrm{K}^b(\proj A)$ will stand for  the bounded homotopy category of finitely generated projective right $A$-modules and ${\mathcal{D}^b(A)}$ is the bounded derived category of finitely generated  right $A$-modules. For a $A$-module $M$, $|M|$ is the number of pairwise non-isomorphic direct summands of $M$.  All modules
considered  basic.

\vspace{0.5cm}
\section{Preliminaries}
\subsection{Tilted algebras}

Let $A$ be an algebra.  An $A$-module $T$ is called $tilting$ if $(1)$ the projective dimension of  $T$  is at most one, $(2)$$\Ext^1_A(T,T) = 0$ and $(3)$ $|T|=|A|$. The endomorphism algebra of a tilting module over a hereditary algebra is called a $tilted$ algebra \cite{HR}.  The following result is very useful.
\begin{theorem}{\rm \cite{H1988}}
Let $H$ be a hereditary algebra, $T$ a tilting $H$-module and $B$=$\End_H(T)$ the corresponding tilted algebra. Then we have

\begin{enumerate}
\item[(1)] The derived functor $\mathrm{RHom}_H(T,-): \mathcal{D}^b(H) \to \mathcal{D}^b(B)$ is an equivalence  which maps  $T$ to $B$.
\item[(2)] $\mathrm{RHom}_H(T,-)$ commutes with the Auslander–Reiten translations and the shifts in the respective categories.
\end{enumerate}
\end{theorem}

\subsection{Silted algebras}

\begin{definition}\label{2.3} {\rm (\cite[Definition 0.1]{AIR})}
Let $T\in\mod A$.
\begin{enumerate}
\item[(1)] $T$ is called {\it $\tau$-rigid} if $\Hom_A(T,\tau T)=0$.
\item[(2)] $T$ is called {\it $\tau$-tilting}  if it is $\tau$-rigid and $|T|=|A|$.
\item[(3)] $T$ is called {\it support $\tau$-tilting} if it is a $\tau$-tilting $A/A eA$-module
for some idempotent $e$ of $A$.
\end{enumerate}
\end{definition}

Sometimes, it is convenient to view support $\tau$-tilting modules and $\tau$-rigid modules as
certain pairs of modules in $\mod A$.

\begin{definition} \label{2.5} {\rm (\cite[Definition 0.3]{AIR})}
Let $(T,P)$ be a pair in $\mod A$ with $P\in \proj A$.
\begin{enumerate}
\item[(1)] $(T, P)$ is called a {\it $\tau$-rigid pair} if $M$ is $\tau$-rigid and $\Hom_A(T,M)=0$.
\item[(2)] $(T, P)$ is called a {\it support $\tau$-tilting pair} if $T$ is $\tau$-rigid and $|T|+|P|=|A|$.
\end{enumerate}
\end{definition}

It was showed in \cite[Proposition 2.3]{AIR} that $(T,P)$ is a support $\tau$-tilting pair in
$\mod A$ if and only if $T$ is a $\tau$-tilting $A/AeA$-module with $eA\cong P$.

 Let $\mathbf{ P}$ be a complex in $\mathrm{K}^b(\proj A)$.  Recall that  $\mathbf{ P}$  is  $silting$ \cite{KV} if $\Hom_{\mathrm{K}^b(\proj A)}(\mathbf{ P},\mathbf{ P}[i]) = 0$ for $i > 0$, and if $\mathbf{ P}$ generates 
 $\mathrm{K}^b(\proj A)$ as a triangulated category. Moreover,  $\mathbf{ P}$ is  called 2-term  if it only has non-zero terms in degree $0$ and $-1$.
 
 The next result show that the relationship between support $\tau$-tilting modules and 2-term silting complexes. For convenience, we denote by $s\tau$-$\tilt A$ all support $\tau$-tilting modules over the algebra $A$ and $2$-$\silt A$  all 2-term silting complexes over $\mathrm{K}^b(\proj A)$ . 
 \begin{theorem}{\rm \cite[theorem 3.2]{AIR}}
 There exists a bijection between $s\tau$-$\tilt A$ and $2$-$\silt A$ given by $(T,P)\in s\tau$-$\tilt A\to P_1\oplus P\to P_0\in2$-$\silt A$ and $\mathbf{ P}\in 2$-$\silt A\to \mathrm{H}^0(\mathbf{ P})\in s\tau$-$\tilt A$, where $P_1\to P_0$ is a minimal projective presentation of $T$. 
 
  \end{theorem}
 
 We call an algebra $A$ is $silted$ \cite{BZ2016} if there is a hereditary algebra $H$ and $\mathbf{ P}\in 2$-$\silt H$ such that $A\cong \End_{\mathcal{D}^b(H)}(\mathbf{ P })$.

\subsection{Cluster-tilted algebras}

 The {\it cluster category} ~$\mathcal{C}_H$ of a hereditary algebra $H$ is  the quotient category $\mathcal{D}^b(H)/F$ where $F$=$\tau^{-1}_{\mathcal{D}}[1]$ and $\tau^{-1}_{\mathcal{D}}$ is the inverse of the Auslander–Reiten translation in $\mathcal{D}^b (H)$.  The space of morphisms from $\tilde{X}$  to $\tilde{Y}$ in $\mathcal{C}_H$  is given by $\Hom_{\mathcal{C}_H}(\tilde{X},\tilde{Y}) $= $\oplus_{i\in \mathbb{Z}}\Hom_{\mathcal{D}^b(H)}(X, F^iY)$.  
It is shown that $\mathcal{C}_H$ is a triangulated category \cite{K2005}.
An object $\tilde{T}\in \mathcal{C}_H$ is called  $tilting$ if $\Ext^1_{\mathcal{C}_H} (\tilde{T}, \tilde{T})$ = $0$  and the number of isomorphism classes of indecomposable summands of $\tilde{T}$ equals $|H|$. The algebra of endomorphisms
$C$ = $\End_{\mathcal{C}_H} (T)$ is called  cluster-tilted \cite{BMR}.  It is proved that the relation-extension of every tilted algebra is cluster-tiled , and every cluster-tilted algebra is of this form in \cite{ABS}.

\section{Main results}

In this section, we prove our main results and give an example to illustrate our results.

 \begin{definition}
 We call a silted algebra $A$ with respect to  $(T, P)$ for some hereditary algebra $H$ if there exists $\mathbf{ P}\in 2$-$\silt H$ which corresponding to $(T,P)\in s\tau$-$\tilt H$ such that $A\cong \End_{\mathcal{D}^b(H)}(\mathbf{ P })$. 
  \end{definition}
  
\begin{proposition}\label{3.2}
 Let $A$ be a silted algebra with respect to  $(T, P)$. Then $A$ is a triangular matrix algebra $\left(\begin{array}{cc} B & 0 \\
M & H_1\end{array}\right)$ where $B$ is a tilted algebra, $H_1$ is a hereditary algebra and $M$ is a $H_1$-$B$-bimodule.
  \end{proposition}
  \begin{proof}
  Suppose that there is a  hereditary algebra $H$ and $\mathbf{ P}\in 2$-$\silt H$ which corresponding to $(T,P)\in s\tau$-$\tilt H$ such that $A\cong End_{\mathcal{D}^b(H)}(\mathbf{ P })$, then we have 
  
  \begin{equation*}
\begin{split}
A&\cong \End_{\mathcal{D}^b(H)}(\mathbf{ P })\\
&\cong  \End_{\mathcal{D}^b(H)}(T\oplus P[1]){\rm( by ~Theorem 2.4)}\\
&\cong  \left(\begin{array}{cc} \End_{\mathcal{D}^b(H)}(T) & \Hom_{\mathcal{D}^b(H)}(P[1],T) \\
\Hom_{\mathcal{D}^b(H)}(T,P[1]) & \End_{\mathcal{D}^b(H)}(P[1])\end{array}\right)
\end{split}
\end{equation*}  
 Take $H'$=$H/HeH$, we have $H'$ is a hereditary algebra , where $eH\cong P$.  Therefore,  $T$ is a tilting $H'$-module and   $B=\End_{\mathcal{D}^b(H)}(T)\cong \End_{H}(T)\cong \End_{H'}(T)$ is a tilted algebra. Moreover, $H_1$=$\End_{\mathcal{D}^b(H)}(P[1])\cong \End_H(P)\cong eHe$ is a hereditary algebra. Note that $\Hom_{\mathcal{D}^b(H)}(P[1],T)$=$0$ since $P$ is projective and $M$=$\Hom_{\mathcal{D}^b(H)}(T,P[1])\cong \Ext^1_H(T,P)$ is a   $H_1$-$B$-bimodule,  we have  $A$ is a triangular matrix algebra. 
   \end{proof}

\begin{lemma}
Let $\mathcal{C}_H$ be a cluster category of a hereditary algebra $H$ and  $T \in \mod H$. Then  we have 
$$\End_{\mathcal{C}_H}(\tilde{T},\tilde{T})\cong\End_{\mathcal{D}^b(H)}(T)\ltimes\Hom_{\mathcal{D}^b(H)}(T, FT).$$
where $\ltimes$ stand for the trivial extension.
 \end{lemma}
\begin{proof}
It follows from {\rm \cite [Lemma 3.3]{ABS}}.
\end{proof}

\begin{theorem}
Let $A$=$\left(\begin{array}{cc} B & 0 \\
M & H_1\end{array}\right)$ be a silted algebra with respect to  $(T, P)$ for some hereditary algebra $H$. Then the matrix algebra $\left(\begin{array}{cc} B\ltimes\Ext^1_H(T,\tau^{-1}_HT) & \Hom_H(P, \tau^{-1}_H T) \\ M & H_1\end{array}\right)$  is a cluster-tilted algebra.
\end{theorem}

  \begin{proof}
  
  Let $A$=$\left(\begin{array}{cc} B & 0 \\
M & H_1\end{array}\right)$ be a silted algebra with respect to  $(T, P)$ for some hereditary algebra $H$.  Then $\tilde{T}\oplus \tilde{P}[1]$ is a cluster-tilting object in $\mathcal{C}_H$. 
For any two $H$-modules $X$ and $Y$, we have $\Hom_{\mathcal{D}^b(H)}(X,Y[i])= 0$ for all $i \geqslant2$ since $H$ is  hereditary. Hence, we have  
  \begin{equation*}
\begin{split}
\End_{\mathcal{C}_H}(\tilde{T}\oplus \tilde{P}[1])
&\cong  \left(\begin{array}{cc} \End_{\mathcal{C}_H}(\tilde{T}) & \Hom_{\mathcal{C}_H}(\tilde{P}[1],\tilde{T}) \\
\Hom_{\mathcal{C}_H}(\tilde{T},\tilde{P}[1]) & \End_{\mathcal{C}_H}(\tilde{P}[1])\end{array}\right)\\
&\cong  \left(\begin{array}{cc} \End_{\mathcal{D}^b(H)}(T)\ltimes\Hom_{\mathcal{D}^b(H)}(T, FT) & \Hom_{\mathcal{C}_H}(\tilde{P}[1],\tilde{T}) \\
\Hom_{\mathcal{C}_H}(\tilde{T},\tilde{P}[1]) & \End_{\mathcal{C}_H}(\tilde{P}[1])\end{array}\right){\rm (by ~Lemma 3.3)}\\
&\cong  \left(\begin{array}{cc} B\ltimes\Ext^1_H(T,\tau^{-1}_HT) & \Hom_{\mathcal{C}_H}(\tilde{P}[1],\tilde{T}) \\
\Hom_{\mathcal{C}_H}(\tilde{T},\tilde{P}[1]) & \End_{\mathcal{C}_H}(\tilde{P}[1])\end{array}\right)\\
&\cong  \left(\begin{array}{cc} B\ltimes\Ext^1_H(T,\tau^{-1}_H) & \Hom_{\mathcal{D}^b(H)}(P[1],FT) \\
\Hom_{\mathcal{D}^b(H)}(T,P[1])  & \End_{\mathcal{D}^b(H)}(P[1])\end{array}\right)\\
&\cong  \left(\begin{array}{cc} B\ltimes\Ext^1_H(T,\tau^{-1}_HT) & \Hom_H(P,\tau^{-1}_HT) \\
M & H_1\end{array}\right)\\
\end{split}
\end{equation*}  
which is a cluster-tilted algebra.
   \end{proof}
  
As a consequence, we have the following result.

\begin{corollary}
Let $A$ be a silted algebra with respect to  $(T, P)$ for some hereditary algebra $H$. If $T$ is injective, then $A$ is hereditary. In particular,  a tilted algebra which induced by a injective tilting module is hereditary. 
\end{corollary}
\begin{proof}
Since $T$ is injective, we have $\tau^{-1}_HT=0$. By  Theorem 3.4, $A$ is a cluster-tilted algebra whose global dimension is at most three. Note that every cluster-tilted algebra is 1-Gorenstein \cite{KR}. Since the projective dimension of every module over a  1-Gorenstein algebra is at most one or infinite,  we get the global dimension  of $A$ is at most one, and so $A$ is  hereditary.
\end{proof}

\begin{theorem}
Let $A$=$\left(\begin{array}{cc} B & 0 \\
M & H_1\end{array}\right)$ be a silted algebra with respect to  $(T, P)$ for some hereditary algebra $H$. If $\Hom_H(P,\tau^{-1}_HT)=0$, then the triangular matrix algebra $\left(\begin{array}{cc} B\ltimes\Ext^2_B(DB,B) & 0 \\
M & H_1\end{array}\right)$ is a cluster-tilted algebra.
\end{theorem}
\begin{proof}
Take $H'=H/HeH$, we have  $\tau^{-1}_HT$ is a $H'$-module since $\Hom_H(P,\tau^{-1}_HT)=0$  where $eH\cong P$. 
Therefore, we have 
  \begin{equation*}
\begin{split}\Ext^1_H(T,\tau^{-1}_HT)
&\cong\Ext^1_{H'}(T,\tau^{-1}_{H'}T)\\
&\cong\Hom_{\mathcal{D}^b(H')}(T, F'T)\\
&\cong \Hom_{\mathcal{D}^b(B)}(B, F''B){\rm (by~ Lemma 2.1)}\\
&\cong \Hom_{\mathcal{D}^b(B)}(\tau_{\mathcal{D}^b(B)}B[1], B[2])\\
&\cong \Hom_{\mathcal{D}^b(B)}(DB, B[2])\\
&\cong \Ext^2_B(DB, B)\\
\end{split}
\end{equation*}  
 where $F'=\tau^{-1}_{{\mathcal{D}^b(H')}}[1]$  and $F''=\tau^{-1}_{{\mathcal{D}^b(B)}}[1]$  is the functor corresponding to $F'$ in the derived category $\mathcal{D}^b(B)$.  
\end{proof}
Note that a tilted algebra is exactly silted algebra with respect to  $(T, 0)$ for some hereditary algebra $H$, we can easy get the following result.
\begin{corollary}
The relation-extension of every tilted algebra is cluster-tilted.
\end{corollary}
\begin{example}
Let $H$ be a hereditary algebra given by the following quiver,
$$
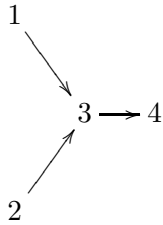
\xymatrix@C=15pt{1\ar[rd]&&\\
&3\ar[r]&4\\
2\ar[ru]&&}$$
The support $\tau$-tilting pair $(T,P)=(P_4\oplus P_1\oplus S_1, P_2)$ corresponding to the 2-term silting complex $0\to P_4\oplus0\to P_1\oplus P_3\to P_1\oplus P_2\to 0$ which induced a silted algebra given as follows,
$$\xymatrix@C=15pt{1&2\ar_{\gamma}[l]&3\ar_{\beta}[l]&4\ar_{\alpha}[l]}$$
with the relations $\alpha\beta=0$ and $\beta\gamma=0$. 
Note that $$dim_k\Ext^1_H(T,\tau^{-1}_HT)=2,  ~~dim_k\Hom_H(P, \tau^{-1}_HT)=1.$$
{\rm(}in fact, $dim_k\Ext^1_H(S_1,\tau^{-1}_HP_4)=1,~~dim_k\Ext^1_H(S_1,\tau^{-1}_HP_1)=1,  ~~dim_k\Hom_H(P_2, \tau^{-1}_HP_1)=1$.{\rm)}
By Theorem 3.4,  we can construct a cluster-tilted algebra given by the following quiver,
$$
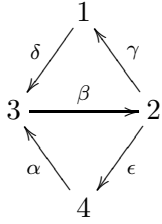
\xymatrix@C=15pt{&1\ar_{\delta}[ld]&\\
3\ar^{\beta}[rr]&&2\ar_{\gamma}[lu]\ar^{\epsilon}[ld]\\
&4\ar^{\alpha}[lu]&}$$
with relations $\gamma\delta=\epsilon\alpha, \alpha\beta=0, \beta\gamma=0,\beta\epsilon=0, \delta\beta=0$.

\end{example}

\subsection*{Acknowledgements}

The author would like to thank  professor Ralf Schiffler and professor Zhaoyong Huang for helpful discussions. He also  thanks  the referee for the useful and detailed suggestions.  This work was partially supported by the National natural Science Foundation of China (No. 11571164).

\end{document}